\documentclass[12pt]{amsart}
\usepackage{amsmath,amsfonts,amssymb,hyperref}

\newtheorem{theorem}{Theorem}

\newtheorem{lemma}[theorem]{Lemma}
\newtheorem{proposition}[theorem]{Proposition}

\newtheorem{corollary}[theorem]{Corollary}

\newtheorem{question}[theorem]{Question}

\newtheorem*{claim}{Claim}

\theoremstyle{definition}

\theoremstyle{remark}

    \def\ll{\langle} \def\rr{\rangle}
    \def\bp{\begin{pmatrix}}
\def\smallsetminus{\setminus} \def\ep{\end{pmatrix}}
\def\bn{\begin{enumerate}}

\def\mid{|}
   \def\en{\end{enumerate}}
\def\ba{\begin{array}} \def\ea{\end{array}}

\def\ker{\operatorname{Ker}}\def\be{\begin{equation}} \def\ee{\end{equation}}

\def\co{\colon}\def\GG{\mathcal{G}} 
\def\PP{\mathcal{P}}\def\HH{\mathcal{H}}

\def\CC{\mathcal{C}}

\def\MM{\mathcal{M}}

\def\PP{\mathcal{P}}

\def\op{\operatorname}
\def\ol{\overline}

\newcommand{\Z}{\mathbb{Z}}

\input xy
\xyoption{all}

\title{The membership problem for $3$-manifold  groups is solvable}

\author{Stefan Friedl}
\address{Fakult\"at f\"ur Mathematik\\ Universit\"at Regensburg\\   Germany}
\email{sfriedl@gmail.com}

\author{Henry Wilton}
\address{Department of Mathematics, University College London, UK}
\email{hwilton@math.ucl.ac.uk}

\begin{document}

\begin{abstract}
We show that the Membership Problem for finitely generated subgroups of 3-manifold groups is solvable.
\end{abstract}

\maketitle

\section{Introduction}

The classical group-theoretic decision problems were formulated by Dehn \cite{dehn_unendliche_1911} in work on the topology of surfaces.  He considered in particular the following questions about finite presentations $\langle A\mid R\rangle$ for a group $\pi$:
\bn
\item the Word Problem, which asks for an algorithm to determine whether or not a word in the generators $A$ represents the trivial element of $\pi$;
\item the Conjugacy Problem, which asks for an algorithm to determine whether or not two words in the generators $A$ represent conjugate elements of $\pi$.
\en
In this context another question  arises naturally:
\bn
\item[(3)] the  Membership Problem, where the goal is to determine whether a given element of a group lies in a specified subgroup.
\en
Note that a solution to the Conjugacy Problem and also a solution to the Membership Problem each give a solution to the Word Problem. 
The initial hope that these problems might always be solvable were dashed by Novikov \cite{No55} and Boone \cite{Bo58} who  showed that there exist finitely presented groups with unsolvable  Word Problem. 

It is therefore natural to ask for which classes of groups the above problems can be solved.
In this paper we will discuss the case of 3-manifold groups, i.e.\  fundamental groups of compact 3-manifolds.
Our understanding of 3-manifold groups has expanded rapidly over the recent years. In particular it is a consequence of the Geometrization Theorem due to Perelman
\cite{Pe02,Pe03a,Pe03b,MT07} and work of Hempel \cite{hempel_residual_1987} that 3-manifold groups are residually finite, which gives rise to a solution to the Word Problem.
Furthermore, by \cite{hamilton_separability_2011}, fundamental groups of orientable $3$-manifolds are `conjugacy separable',
which gives rise to a solution to the Conjugacy Problem (see also Pr\'eaux \cite{Pr06,Pr12} and Sela \cite{Sel93} for an earlier solution.)

The Word Problem and the Conjugacy Problem for fundamental groups of orientable 3-manifold groups can thus be solved by translating the problem to dealing with the corresponding problems for finite groups,
which in turn can be solved trivially. The same approach cannot work for the Membership Problem since $3$-manifold groups are not  `subgroup separable', see e.g.\ \cite{BKS87,NW01}.
The main goal of this paper is to show that the Membership Problem has nonetheless a uniform 
solution  for 3-manifold groups. More precisely, we have the following theorem.

\begin{theorem}\label{mainthm}
There exists an algorithm which takes as input a finite presentation 
$\pi=\ll A\mid R\rr$ of a 3-manifold group, a finite set of words  $w_1,\dots,w_k$ in $A$ and a word $z$ in $A$ and which determines whether or not the element  $z\in \pi$ lies in the subgroup of $\pi$ generated by $w_1,\dots,w_k$. 
\end{theorem} 

The four main ingredients in the proof are the following:
\bn
\item The resolution of the Tameness Conjecture by Agol \cite{Ag04} and Calegari--Gabai \cite{CG06}.
\item An algorithm of Kapovich--Weidmann--Miasnikov
 \cite{KWM05} which deals with the Membership Problem of the fundamental group of a graph of groups.
\item Algorithms of Jaco--Letscher--Rubinstein \cite{JLR02}, Jaco--Rubinstein \cite{jaco_efficient_2003} and Jaco--Tollefson \cite{jaco_algorithms_1995}
which determine the prime decomposition and the JSJ decomposition of a given triangulated 3-manifold.
\item The Virtually Compact Special Theorem of Agol \cite{Ag13} and Wise \cite{Wi12a,Wi12b}.
\en

The paper is organized as follows.  In Section \ref{section:wordandconjugacy} we recall the solutions to the Word Problem and the Conjugacy Problem for 3-manifold groups. In Section \ref{section:mainthm} we discuss our main theorem in more detail.  In Section \ref{section:basicalgorithms} we recall several basic results and algorithms, and we give the proof of our main theorem in Section \ref{section:proofmp}.  Finally, in Section \ref{section:alternative} we will quickly discuss a slightly different approach to the proof of our main theorem and we will raise a few questions.

\subsection*{Conventions.} All $3$-manifolds and surfaces are assumed to be compact
 and connected.  We furthermore assume that all graphs are connected. Finally we assume that all classes of groups are closed under isomorphism.

\section{The Word Problem and the Conjugacy Problem for 3-manifold groups}
\label{section:wordandconjugacy}

We start out with introducing several definitions which we will need throughout this paper.
\bn
\item
Given a set $A$ we denote by $F(A)$ the free group generated by $A$.
As usual we will freely go back and forth between words in $A$ and elements represented by these words in $F(A)$.
\item
A \emph{finite presentation} $\ll A\mid R\rr$ is  a finite set $A$ together 
with a finite set $R$ of elements in $F(A)$.  We follow the usual convention that by $\ll A\mid R\rr$ 
 we indicate at the same time the 
finite data and also the group 
\[ F(A)/\ll \ll R\rr \rr,\]
i.e.\ the quotient of $F(A)$ by the normal closure $\ll \ll R\rr\rr$ of $R$ in $F(A)$. In the notation we will for the most part  not distinguish between elements in $F(A)$ and the elements they represent in $\ll A\mid R\rr$.
\item A \emph{finite presentation for a group $\pi$} is a finite presentation $\ll A\mid R\rr$ such that $\pi\cong \ll A\mid R\rr$. 
We say that a group $\pi$ is \emph{finitely presentable} if it admits a finite presentation.
\en

Before we start with the solution to the Membership Problem for 3-manifold groups
it is worth looking at the solution to the Word Problem and to the Conjugacy Problem.

Recall that a group $\pi$ is called \emph{residually finite} if given any non-trivial $g\in \pi$ there exists a homomorphism $f\colon \pi\to G$ to a finite group $G$ such that $f(g)$ is non-trivial.
It is  consequence 
of the Geometrization Theorem \cite{Th82,Pe02,Pe03a,Pe03b} and of  work of Hempel \cite{hempel_residual_1987}
that 3-manifold groups are residually finite.

The following well-known lemma thus gives a solution to the Word Problem for 3-manifold groups.

\begin{lemma}\label{lem:resfinite}
There exists an algorithm which takes as input  a finite presentation $\pi=\ll A\mid R\rr$ 
and an element $w\in F(A)$, and which, if $\pi$ is residually finite, determines whether or not $w$  represents the trivial element.
\end{lemma}

\begin{proof}
Let $\pi=\ll A\mid  R\rr$  be a finitely presented group and let $w\in F(A)$. We suppose that $\pi$ is residually finite.
If $w$ represents the trivial element, then we will see this eventually by writing down systematically all words which are products of conjugates of elements in $R\cup R^{-1}$.  On the other hand, if $w$ does not represent the trivial word, then by residual finiteness there exists a homomorphism $f\colon \pi\to G$ to a finite group with $f(g)\ne e$. 
  Since $\ll A\mid  R\rr$ is a finite presentation we can recursively enumerate all homomorphisms from $\pi$ to finite groups. 
  After finitely many steps we will thus  detect that $g$ is indeed non-trivial.
\end{proof}

As we mentioned in the introduction, 
Pr\'{e}aux, extending Sela's work on knot groups \cite{Sel93}, proved that the Conjugacy Problem is solvable for the fundamental groups of orientable \cite{Pr06} and non-orientable \cite{Pr12} 3-manifolds.  (Note that, in contrast to many other group properties, solvability of the Conjugacy Problem does not automatically pass to finite extensions \cite{CM77}.)  

It is natural to ask whether there also exists a solution to the Conjugacy Problem along the lines of Lemma \ref{lem:resfinite}.  In the following we say that a group $\pi$ is \emph{conjugacy separable} if, given any non-conjugate $g,h\in \pi$, there exists a homomorphism $f\colon \pi\to G$ to a finite group $G$ such that $f(g)$ and $f(h)$ are non-conjugate.  A slight variation on the proof of Lemma \ref{lem:resfinite} also shows that the Conjugacy Problem is solvable if the given group is  conjugacy separable.

Hamilton, the second author and Zalesskii,
 building on the recent work of Agol \cite{Ag13} and Wise \cite{Wi12a,Wi12b}
  and work of Minasyan \cite{Min12} showed that fundamental groups  of  orientable 3-manifold $N$ are conjugacy separable. This result gives another solution to the  Conjugacy Problem for fundamental groups of \emph{orientable} 3-manifolds.  
  
\section{The statement of the main theorem}\label{section:mainthm}

Let $\CC$ be  a class of finitely presentable groups.  We say that the \emph{Membership Problem is solvable in $\CC$}  if there exists an algorithm which takes as input a finite presentation  $\ll A\mid R\rr$, a finite set of words  $w_1,\dots,w_k$ in $A$ and a word $z$ in $A$ and which, if $\ll A\mid R\rr$ is a presentation for a group $\pi$ in $\CC$, determines whether or not the element  $z\in \pi$ lies in the subgroup  generated by $w_1,\dots,w_k\in \pi$.

The following theorem is now a reformulation of our main theorem.

\begin{theorem} \label{thm: MP}\label{thm:mp}
The Membership Problem is solvable for the class of 3-manifold groups.
\end{theorem} 

In Section \ref{section:wordandconjugacy} we saw that separability properties of fundamental groups can be used to solve the Word Problem and the Conjugacy Problem for fundamental groups of (orientable) 3-manifold. 

The `right' notion of separability in the context of the Membership Problem is the separability of subgroups.
More precisely, in the following we say that a subset $\Gamma$ of a group $\pi$ is \emph{separable} if given any $g\not\in \Gamma$ there exists a homomorphism $f\colon\pi\to G$ to a finite group such that $f(g)\not\in f(\Gamma)$.
We say that a group $\pi$ is \emph{subgroup separable} if every finitely generated subgroup is separable.
The proof of Lemma \ref{lem:resfinite} can easily be modified to show that the Membership Problem is solvable for the class of finitely presentable groups which are subgroup separable.

Scott \cite{Sc78} showed that the fundamental groups of Seifert fibred
 3-manifolds are subgroup separable. Furthermore, it follows from 
work of Agol \cite{Ag13} and Wise \cite{Wi12a,Wi12b}, together with the proof of the Tameness Conjecture by Agol \cite{Ag04} and Calegari--Gabai \cite{CG06}
that fundamental groups of hyperbolic 3-manifolds are subgroup separable.
(The precise references for this statement can be found in \cite{AFW12}.)

On the other hand, there are many examples of fundamental groups of prime 3-manifolds which are not subgroup separable, see e.g. \cite{BKS87,NW01}.
 We thus see that we cannot hope to prove Theorem \ref{thm: MP} in the general case
by appealing to separability properties only. 

The key idea in the proof of Theorem \ref{thm:mp}
is to apply a theorem of Kapovich--Weidmann--Miasnikov \cite{KWM05}
which provides a solution to the  Membership Problem 
for fundamental groups of graphs of groups if various conditions are satisfied.
We will apply this theorem twice, once to reduce the problem to the case of prime 3-manifolds, and then later on to deal with the case of prime 3-manifolds with non-trivial JSJ decomposition.

In Section  \ref{section:basicalgorithms} we will first make some preliminary observations.
In Section \ref{section:kmw} we formulate the aforementioned main theorem of \cite{KWM05}
which we will use in Section \ref{section:proofmpi} to  argue that it suffices to prove our main theorem for  closed, orientable, prime 3-manifolds.
In Sections \ref{section:jsj}, \ref{section:subgroupshypsfs} and  \ref{section:proofcomputegeneratingssets} we will show that one can use 
the main theorem of \cite{KWM05} to deal with the fundamental groups of closed, orientable, prime 3-manifolds with non-trivial JSJ decomposition.

\section{Preliminary results}\label{section:basicalgorithms}

\subsection{Basic algorithms}

We start out with several basic lemmas which we will need time and again during the paper. 
The statements of the lemmas are well-known to experts, but we include proofs for the reader's convenience.
At a first reading of the paper it might nonetheless be better to skip this section.

\begin{lemma}\label{lem:findiso}
There exists an algorithm which takes as input  two finite presentations  $\ll A\mid  R\rr$ and $\ll A'\mid  R'\rr$
and which  finds an isomorphism  $\ll A\mid  R\rr\to \ll A'\mid  R'\rr$, if  such an isomorphism exists.
\end{lemma}

Here, by `finds an  isomorphism  $\ll A\mid  R\rr\to \ll A'\mid  R'\rr$' we mean that 
the algorithm finds a map $A\to F(A')$ which descends to an isomorphism
 $\ll A\mid  R\rr\to \ll A'\mid  R'\rr$.

\begin{proof}
Let $\pi=\ll A\mid  R\rr$ and $\pi'=\ll A'\mid  R'\rr$ be finite presentations.
We denote by $T$ the subgroup of $F(A)$ normally generated by $R$.
Similarly we define $T'$.

Note that a homomorphism $\varphi\co \ll A\rr \to \ll A'\rr$ descends to an isomorphism $\varphi\co \pi\to \pi'$ if and only if the following two conditions hold:
\bn
\item for all $r\in R$ we have $\varphi(r)\in T'$, and
\item there exists a
homomorphism $\psi\co \ll A'\rr \to \ll A\rr$ such that $\psi(r')\in T$ for all $r'\in R'$, such that $\psi(\phi(g))g^{-1}\in T$ for all $g\in A$ and  such that $\phi(\psi(g'))(g')^{-1}\in T'$ for all $g'\in A'$.
\en

We run the following Turing machines simultaneously:
\bn
\item A Turing machine which produces a list of all words in $A$ of $\pi$ which represent the trivial word, i.e.\ a Turing machine which produces a list of all elements of $T$.
\item A Turing machine which  produces a list of all homomorphisms $\varphi\co \ll A\rr\to \ll A'\rr$, i.e.\ a Turing machine which outputs all $\mid  A\mid  $-tuples of elements in $\langle A'\rangle$. 
\item A Turing machine which produces a list of all words in $A'$ of $\pi'$ which represent the trivial word.
\item A Turing machine which  produces a list of all homomorphisms $\varphi\co \ll A'\rr\to \ll A\rr$.
\en
Now suppose  there exists an isomorphism $\pi\to \pi'$.
It is  clear that after finitely many steps we will find a pair
 $\varphi\co \ll A\rr \to \ll A'\rr$ and  $\psi\co \ll A'\rr \to \ll A\rr$ such that the following conditions hold:
\bn 
 \item for all $r\in R$ we have $\varphi(r)\in T'$, 
 \item for all $r'\in R'$ we have $\psi(r')\in T$
 \item for all $g\in A$ we have  $\psi(\phi(g))g^{-1}\in T$, and
 \item for all $g'\in A'$ we have  $\phi(\psi(g'))(g')^{-1}\in T'$.
 \en
\end{proof}

In the following we say that a subgroup $\Gamma$  of a group $\pi$ is a \emph{retract}
if there exists a  retraction $r\colon \pi\to \Gamma$, i.e.\ a homomorphism with $r(g)=g$ for all $g\in \Gamma$.
Almost the same argument as in Lemma \ref{lem:findiso} also proves the following lemma.

\begin{lemma}\label{lem:findretract}
There exists an algorithm  which takes as input  two finite presentations  $\pi=\ll A\mid  R\rr$ and $\pi'=\ll A'\mid  R'\rr$ and which finds a map $f\colon \pi\to \pi'$ and a left-inverse $g\colon \pi'\to \pi$ to $f$,
if $\pi$ is isomorphic to a retract of $\pi'$.
\end{lemma}

We also have the following lemma.

\begin{lemma}\label{lem:samesubgroup}
There exists an algorithm which takes as input a finite presentation  $\pi=\ll A\mid R\rr$
and two finite set of words $X$ and $Y$ in $A$, and which certifies that the two sets $X$ and $Y$ generate the same subgroup of $\pi$, if this is indeed the case.
\end{lemma}

\begin{proof}
The algorithm enumerates all elements  in $\ll X,\ll \ll R\rr\rr\,\rr\subset F(A)$
and it enumerates all elements in $\ll Y,\ll \ll R\rr\rr\,\rr\subset F(A)$. Note that $X$ and $Y$ generate the same subgroup 
of $\pi$ if and only if 
$X\subset \ll Y,\ll\ll R\rr\rr\,\rr$ and $Y\subset \ll X,\ll \ll R\rr\rr\,\rr$. If this is indeed the case, then this will be verified after finitely many steps.
\end{proof}

\begin{lemma}\label{lem:finiteindexkernelcokernel}
There exists an algorithm which takes as input   two finite presentations  $\Gamma=\ll A\mid  R\rr$ and $\pi=\ll B\mid  S\rr$,
 a homomorphism  $f\colon \Gamma\to \pi$ and
a  set of elements in $\pi$ which generate a finite-index subgroup $\pi_0\subset \pi$, and which gives as output a  set of coset representatives for $\Gamma_0:=f^{-1}(\pi_0)$ in $\Gamma$ and
a finite presentation $\ll A_0\mid R_0\rr$ together with an isomorphism $\ll A_0\mid R_0\rr\to \Gamma_0$. 
\end{lemma}

\begin{proof}

We start out with the following claim.\\

\noindent {\it{Claim.}}
There exists an algorithm which takes as input a finite presentation 
 $\pi=\ll B\mid S\rr$ and a finite set $X$ of elements in $\pi$
such that $\pi_0=\ll X\rr$ is  a finite-index subgroup of $\pi$,
and which  finds a homomorphism $f\colon \pi\to G$ to a finite group $G$
and a subgroup $G_0\subset G$ such that $\pi_0=f^{-1}(G_0)$. \\

First note that there exists indeed a homomorphism $f\colon \pi\to G$ onto a finite group and a subgroup $G_0$ of $G$
such that $\pi_0=f^{-1}(G_0)$. 
For example, we could take $G$ to be the quotient of $\pi$ by the core of $\ll X\rr$ and $G_0$ the image of $\pi_0$ in this quotient.

The algorithm now goes through all epimorphisms from $\pi$ to finite groups $G$.
For each epimorphism $f\colon \pi \to G$ onto a finite group we also consider all finite-index subgroups $G_0\subset G$. We then calculate a set  of coset representatives for $G_0\subset G$.
By enumerating the elements in $\pi_0$ and determining their images under $f$
we can then  find preimages of the coset representatives, i.e.\ we can find a set of coset representatives for $f^{-1}(G_0)$. Using the Reidemeister-Schreier process we can then determine
a finite set of generators for $f^{-1}(G_0)$.
If $f^{-1}(G_0)$ and $\pi_0$ generate the same group, then this will be certified by the algorithm
of Lemma \ref{lem:samesubgroup}. By the above discussion this algorithm will terminate after finitely many steps.
This concludes the proof of the claim.

Now suppose we are given two finite presentations  $\Gamma=\ll A\mid  R\rr$ and $\pi=\ll B\mid  S\rr$, a homomorphism  $f\colon \Gamma\to \pi$ and
a  set of elements in $\pi$ which generate a finite-index subgroup $\pi_0\subset \pi$.
We apply the above claim to $\pi_0\subset \pi$.  
We  write  $H:=(g\circ f)(\Gamma)$. Note that  $H_0:=(g\circ f)(\Gamma_0)=(g\circ f)(\Gamma)\cap f(\pi_0)$.  
We can evidently find coset representatives for $H_0\subset H$. 
As in the proof of the claim we can furthermore find preimages of the coset representatives under the map $g\circ f$, which are then coset representative for $\Gamma_0\subset \Gamma$.
Using the Reidemeister--Schreier process we can now find 
a finite presentation $\ll A_0\mid R_0\rr$ together with an isomorphism $\ll A_0\mid R_0\rr\to \Gamma_0$. 
\end{proof}

\begin{lemma}\label{lem:allfiniteindexsubgroups}
There exists an algorithm which, given a finite presentation $\pi=\ll A\mid R\rr$, 
determines a list of all finite-index subgroups; more precisely, it provides
a list of finite presentations $\pi_i=\ll A_i\mid R_i\rr$ and monomorphisms $f_i\colon \pi_i\to \pi$
such that any finite-index subgroup  of $\pi$ agrees with $f_i(\pi_i)$ for some $i$.
\end{lemma}

\begin{proof}
Given a finite presentation $\pi=\ll A\mid R\rr$ we can list all homomorphisms to finite groups.
Furthermore, given a homomorphism $\pi\to G$ to a finite group and a subgroup $G_0\subset G$
we saw in the previous proof that we can determine coset representatives for $f^{-1}(G_0)\subset \pi$,
and the Reidemeister--Schreier procedure gives a finite presentation for $f^{-1}(G_0)$ together with a map to $\pi$. As we saw in the proof of the previous lemma, any finite index subgroup of $\pi$ arises that way.
\end{proof}

\begin{lemma}\label{lem:normal}
There exists an algorithm which given a  finite presentation $\Gamma=\ll A\mid R\rr$
and a finite set of elements $X\subset F(A)$
certifies that  the subgroup $\langle X\rangle\subset \Gamma$ is normal,
if this is the case.
\end{lemma}

\begin{proof}
Let $\Gamma=\langle a_1,\dots,a_l\mid r_1,\dots,r_m\rangle$ be a finite presentation and let $X=\{g_1,\dots,g_k\}\subset F(A)$.
Note that $g_1,\dots,g_k$ generate a normal subgroup in $\Gamma$ if and only 
if $a_ig_ja_i^{-1}$ lie in   $\ll X,\,\ll \ll R\rr\rr\,\rr $ for any $i$ and $j$. 

But if this is the case, then this can be certified by enumerating all
elements in $\ll X,\,\ll \ll R\rr\rr\,\rr $ and after finitely many steps we will
have verified that indeed  all the elements $a_ig_ja_i^{-1}$ lie in  $\ll X,\,\ll \ll R\rr\rr\,\rr $.
\end{proof}

\begin{lemma}\label{lem:givesdesiredquotient}
There exists an algorithm which given  two finitely presented groups  $\Gamma=\ll A\mid  R\rr$ and $\pi=\ll B\mid  S\rr$
and   a finite set of elements $X\subset \Gamma$
certifies  that  the subgroup $\langle X\rangle\subset \Gamma$ is normal,
and that the quotient of $\Gamma$ by the normal subgroup $\langle X\rangle$ is isomorphic to $\pi$.
\end{lemma}

\begin{proof}
Let $X=\{g_1,\dots,g_k\}$ be a finite set of elements in a finite presentation
$\Gamma=\langle a_1,\dots,a_l\mid r_1,\dots,r_m\rangle$.
We first apply the algorithm from the proof of Lemma \ref{lem:normal}
which allows us to certify that  $\langle g_1,\dots,g_k\rangle\subset \Gamma$ is normal.
If  $\Gamma/\langle g_1,\dots,g_k\rangle$ is isomorphic to $\pi$,
then the group given by the finite presentation
\[ \langle a_1,\dots,a_l\mid r_1,\dots,r_m,g_1,\dots,g_k\rangle\]
 is isomorphic to $\pi$. But by Lemma \ref{lem:findiso} the existence of such an isomorphism can be certified after finitely many steps.
\end{proof}

\subsection{Preliminary observations on the Membership Problem}\label{section:prelims}

In this section we will prove two elementary lemmas dealing with the Membership Problem.

\begin{lemma} \label{lem:mpfiniteindex}
There exists an algorithm which takes as input a finite presentation $\pi=\ll A\mid R\rr$,
a finite set $X\subset \pi$ such that $\ll X\rr$ is a finite-index subgroup of $\pi$
and an element $g\in \pi$ and   decides whether or not $g\in \pi$  lies in $\langle X\rangle$.
\end{lemma}

\begin{proof}
By the proof of Lemma \ref{lem:finiteindexkernelcokernel} we can find
a homomorphism $f\colon \pi\to G$ to a finite group and a  subgroup $G_0\subset G$ such that $\pi_0=f^{-1}(G_0)$.
Given an element $g\in \pi$ we now only have to determine whether or not $f(g)$ lies in $G_0$, which can be done trivially.
\end{proof}

Before we state our next lemma we need to give two more definitions:
\bn
\item We say that a group $\pi$ is  \emph{virtually isomorphic to a retract} of a group $\Gamma$ if there exists a finite-index subgroup $\pi_0$ of $\pi$  which is isomorphic to a retract of  $\Gamma$.
\item  A class $\CC$ of groups is \emph{recursively enumerable} if there exists a Turing machine that outputs a list of finite presentations such that any group in $\CC$ is isomorphic to the group defined by one of those finite presentations.
\en
We can now formulate the following lemma.

\begin{lemma}\label{lem: MP passing to fi subgroups}\label{lem:mpvirtualretract}
Let $\mathcal{C},\mathcal{D}$ be classes of finitely presentable groups, and suppose that every group in $\mathcal{C}$ is virtually isomorphic to a retract of a group in $\mathcal{D}$.
  If
\begin{enumerate}
\item the Membership Problem is  solvable in $\mathcal{D}$, and 
\item  $\mathcal{D}$ is recursively enumerable,
\end{enumerate}
then the Membership Problem is also  solvable in $\mathcal{C}$.
\end{lemma}

\begin{proof}
Let $\pi=\ll A\mid R\rr$ be a finite presentation of a group in $\mathcal{C}$.    Using Lemma \ref{lem:allfiniteindexsubgroups} we may enumerate all subgroups of finite index in $\pi$.  Because $\mathcal{D}$ is recursively enumerable, we will 
by Lemma \ref{lem:findretract} eventually find a finite-index subgroup $\pi_0$ of $\pi$,
a finite  presentation $\ll B\mid S\rr$ for a group $\Gamma$ $\mathcal{D}$, a homomorphism $f\colon \pi_0\to \Gamma$
and a left-inverse $g\colon \Gamma\to \pi_0$ to $f$. Note that $f$ is in particular injective. 

Let $H$ be a finitely generated subgroup of $\pi=\ll A\mid R\rr$, specified by a finite set of elements. By Lemma \ref{lem:finiteindexkernelcokernel}, we may compute a generating set for $H_0= H\cap\pi_0$ and a set of coset representatives $h_1,\ldots, h_k$ for $H_0$ in $H$.  Now, if $g\in \pi$ then, for each $i$, we may by Lemma \ref{lem:mpfiniteindex} determine whether or not $h^{-1}_ig\in\pi_0$.  If there is no such $i$ then evidently $g\notin H$.  Otherwise, if $h^{-1}_ig\in\pi_0$ then, using the solution to the membership problem in $f(\pi_0)\subset \Gamma$, we may determine whether or not $h^{-1}_ig\in H_0$. 
The lemma now follows from the observation that  $g\in H$ if and only if $h^{-1}_ig\in H_0$ for some $i$.
\end{proof}

\section{The Membership Problem for 3-manifold groups}\label{section:proofmp}

\subsection{The theorem of Kapovich--Weidmann--Miasnikov} 
\label{section:kmw}

We first recall that a graph of groups $\GG$ is a finite graph 
with vertex set $V=V(\GG)$ and edge set $E=E(\GG)$  together 
with vertex groups $G_v, v\in V$ and edge groups $H_e,e\in E$ and monomorphisms $i_e\colon H_e\to G_{i(e)}$ and $t_e\colon H_e\to G_{t(e)}$.
Also recall that Serre \cite{Ser80}
associated to a graph $\GG$ of finitely presented groups a finitely presented
 group, which is referred to as the \emph{fundamental group of $\GG$}. 

We also need the following definitions:
\bn
\item 
A \emph{decorated group} is a pair $(G,\{H_i\}_{i\in I})$ where $G$ is a group and $\{H_i\}_{i\in I}$ is a finite set of subgroups of $G$.
We refer to $G$ as the  \emph{vertex group} and to each $H_i, i\in I$ as an \emph{edge group}.
\item We say that two decorated groups  $(G,\{H_i\}_{i\in I})$ and
 $(K,\{L_j\}_{j\in J})$ are \emph{isomorphic}, written as 
 \[ (G,\{H_i\}_{i\in I})\cong  (K,\{L_j\}_{j\in J}), \]
  if there exists an isomorphism $\varphi\colon G\to H$, a bijection $\psi\colon I\to J$ and elements $k_j,j\in J$ such that $\varphi(H_i)=g_{\psi(i)}L_{\psi(i)}g_{\psi(i)}^{-1}$ for any $i\in I$.
\item A \emph{subdecoration} of a decorated group  $(G,\{H_i\}_{i\in I})$ is a pair  $(G,\{H_j\}_{j\in J})$,
where $J\subset I$ is a subset.
\item A \emph{finite decorated group presentation}  is a pair
\[ (\ll A\mid R\rr,\{\ll X_i\mid S_i\rr, f_i\}_{i\in I})\]
where $\ll A\mid R\rr$ is a finite presentation, where $I$ is a finite set,
and  each  $\ll X_i\mid S_i\rr$ is a finite presentation
and each $f_i\colon \ll X_i\mid S_i\rr\to \ll A\mid R\rr$ is a monomorphism.
Note that a finite decorated group presentation  $(\ll A\mid R\rr,\{\ll X_i\mid S_i\rr, f_i\}_{i\in I})$ defines 
a decorated group $(\ll A\mid R\rr,\{f_i(\ll X_i\mid S_i\rr)\}_{i\in I})$.
\item A \emph{finite presentation for a decorated group $(G,\{H_i\}_{i\in I})$} is a finite decorated group presentation $(\ll A\mid R\rr,\{\ll X_i\mid S_i\rr, f_i\}_{i\in I})$ such that 
\[ (G,\{H_i\}_{i\in I})\cong (\ll A\mid R\rr,\{f_i(\ll X_i\mid S_i\rr)\}_{i\in I}). \]
We say that a decorated group  is \emph{finitely presentable} if it admits a finite decorated group presentation.
\en

We say that a class $\PP$ of decorated groups is \emph{recursively enumerable} if there exists a Turing machine that outputs a list of finite decorated group  presentations such that any decorated group  in $\PP$ is isomorphic to a decorated group defined by one of those finite decorated group  presentations.

Finally, given a class $\PP$ of decorated groups we say that a graph of groups with vertex groups $\{G_v\}_{v\in V}$ and edge groups $\{H_e\}_{e\in E}$ 
is \emph{based on $\PP$} if given any vertex $v$ the 
pair $(G_v,\{ i_e(H_e)\}_{e\in E, i(e)=v})$ is equivalent  to a subdecoration 
 of a  decorated group  in $\PP$. 

The following theorem is basically due to Kapovich--Weidmann--Miasnikov.

\begin{theorem}\textbf{\emph{(Kapovich--Weidmann--Miasnikov)}}\label{thm: KWM}\label{thm:kmw05}
\label{thm:kmw}
Let $\PP$ be a class of decorated groups such that the following hold:
\bn
\item[(I)] There is an algorithm which for each  finite decorated group presentation 
$(\ll A\mid R\rr,\{\ll X_i\mid S_i\rr, f_i\}_{i\in I})$ of a  decorated group in $\PP$, each $i\in I$, and each finite set $Y\subset F(A)$
determines whether or not a given element in $\ll A\mid R\rr$ lies in the double coset $\langle f_i(X_i)\rangle \langle Y\rangle\subset \ll A\mid R\rr$.
\item[(II)] The Membership Problem is solvable for the class of all vertex groups appearing in $\PP$.
\item[(III)] Every edge group is \emph{slender},  meaning that every subgroup of an edge group is finitely generated.
\item[(IV)] The Membership Problem is solvable for the class of all edge groups appearing in $\PP$.
\item[(V)] There is an algorithm which for 
 each finite decorated group presentation 
$(\ll A\mid R\rr,\{\ll X_i\mid S_i\rr, f_i\}_{i\in I})$  of a decorated group in $\PP$, each $i\in I$ and each finite set $Y\subset F(A)$ 
  computes a finite generating set for the intersection $\ll f_i(X_i)\rr \cap \ll Y\rr\subset \ll A\mid R\rr$. 
\item[(VI)] The class $\PP$ is recursively enumerable.
\en
Then the Membership Problem is solvable for the class of fundamental groups of all
 graphs of groups based on $\PP$.
\end{theorem}

\begin{proof}
In this proof we assume some familiarity with \cite{KWM05}.  We start out with the following claim.

\begin{claim}
If  a class of decorated groups $\PP$ satisfies Condition I, then it also satisfies
\bn
\item[(I')] There is an algorithm which for each  finite decorated group presentation 
$(\ll A\mid R\rr,\{\ll X_i\mid S_i\rr, f_i\}_{i\in I})$ of a  decorated group in $\PP$, each $i\in I$, and each finite set $Y\subset F(A)$
determines whether or not for a given element $a$ we have  $I:=a\langle f_i(X_i)\rangle\cap  \langle Y\rangle\ne \emptyset$, and if $I\ne \emptyset$, gives as output an element in $I$.
\en
\end{claim}

First note that if $G$ and $H$ are subgroups of a group $\pi$ and if $b\in \pi$, then it follows trivially
that $b\in G\cdot H:=\{ gh\,|\, g\in G,h\in H\}$ if and only if $bG\cap H\ne \emptyset$.
So if $\PP$ satisfies Condition I, then we can determine whether $I:=a\langle f_i(X_i)\rangle\cap  \langle Y\rangle\ne \emptyset$. Now suppose that $I\ne \emptyset$. We now have to find an element in $I$. This means that we have to find an element in
\[ \langle af_i(X_i),\ll \ll aR\rr\rr\,\rangle\cap  \langle Y,\ll \ll R\rr\rr\,\rangle \subset F(A).\]
But we can just enumerate all elements in $\langle af_i(X_i),\ll \ll aR\rr\rr\,\rangle$ and we can enumerate all elements in $\langle Y,\ll \ll R\rr\rr\,\rangle$, and since the intersection is non-trivial we will eventually find an element which lies in both sets.
This concludes the proof of the claim.

It now follows from Conditions I', II, IV and V that any graph of finitely presented groups based on $\PP$ is `benign' in the sense of \cite[Definition~5.6]{KWM05}.
The solvability of the Membership Problem  follows from  \cite[Theorem 5.13]{KWM05} combined with the following claim.

\begin{claim}
There exists an algorithm which takes as input  a finite presentation  $\ll A\mid R\rr$
for the fundamental group of a graph of groups in $\CC$ and which gives as output a finite presentation
of a graph of groups based on $\PP$ together with an isomorphism from $\ll A\mid R\rr$ to the  fundamental group of the graph of groups.
\end{claim}

Let $\ll A\mid R\rr$ be a finite presentation for the fundamental group of a graph of groups in $\CC$.
By assumption the class $\PP$ of decorated groups is recursively enumerable.
It is clear that one can then also recursively enumerate the class of all decorated groups which are subdecorations  of decorated groups in $\PP$. 
It furthermore follows from Lemma \ref{lem:findiso} that  the class of isomorphisms between edge groups in $\PP$ is also recursively enumerable.
 It is now straightforward to see that  the class of  graphs of groups based on $\PP$ is also recursively enumerable. 
By Lemma \ref{lem:findiso} we will eventually find an isomorphism from $\ll A\mid R\rr$
to the finite presentation of the fundamental group of a  graph of groups based on $\PP$.

\end{proof}

 We then have the following immediate corollary
to Theorem \ref{thm:kmw}. See also  \cite[Corollary 5.16]{KWM05}.

\begin{corollary}\label{cor:solvablempproblemandfreeproducts}
Let $\CC$ be a recursively enumerable class of groups 
for which the Membership Problem is solvable.
Then the Membership Problem is solvable for the class of groups which are isomorphic to  finite free products of groups in $\CC$.
\end{corollary}

\begin{proof}
We denote by $\PP$ the class of decorated groups $(G,\{e\})$ with $G\in \CC$. 
It follows easily from our assumptions that Conditions I to VI of Theorem \ref{thm:kmw05} are satisfied for $\PP$.
By Theorem \ref{thm:kmw05} the Membership Problem is solvable for the class of groups which are isomorphic to fundamental groups of graphs of groups based on $\PP$.
The corollary now follows from the observation that if the underlying graph is a tree, then the corresponding
 fundamental group is in this case just the free product of the vertex groups.
\end{proof}

\subsection{The reduction to the case of closed, orientable,  prime 3-manifolds} \label{section:proofmpi}

The goal of this section is to prove the following proposition.

\begin{proposition}\label{prop:2}
If the Membership Problem is solvable for the class of  fundamental groups 
of all closed, orientable,  prime $3$-manifolds, then it is also solvable for the class of fundamental groups of all  $3$-manifolds.
\end{proposition}

In the following we say that a class $\MM$ of $3$-manifolds is \emph{recursively enumerable} if there exists a Turing machine that outputs a list of finite simplicial spaces such that any manifold in $\MM$ is  homeomorphic  to one of those simplicial spaces.
In the proof of Proposition \ref{prop:2} we will need the following theorem
which is basically a combination of work of Moise  \cite{Mo52,Mo77}
and 
Jaco--Rubinstein \cite{jaco_efficient_2003} or alternatively Jaco--Tollefson \cite{jaco_algorithms_1995}.

\begin{theorem}\label{theorem: 3-manifolds are re}
The class of closed, orientable, prime 3-manifolds is recursively enumerable.
\end{theorem}

\begin{proof}
By  Moise's Theorem \cite{Mo52,Mo77} every $3$-manifold $N$
admits a finite triangulation, i.e.\ $N$ can be written as a finite simplicial complex. 

Note that there exists an algorithm which checks whether or not  a given finite simplicial complex 
represents a closed orientable 3-manifold.  Indeed, one only needs to check whether the link of each vertex is a 2-sphere,
and whether the third homology group is non-zero.
Furthermore, by work of 
Jaco--Rubinstein \cite{jaco_efficient_2003} and also by Jaco--Tollefson \cite[Algorithm 7.1]{jaco_algorithms_1995}
there exists an algorithm which given a triangulated closed, orientable 3-manifold $N$ determines whether or not the manifold is prime.

We thus go through all finite simplicial complexes and we keep the ones which represent closed, prime 3-manifolds. By Moise's Theorem any 3-manifold will eventually appear in this list of finite simplicial complexes.
\end{proof}

\begin{proof}[Proof of Proposition \ref{prop:2}]
Let $N$ be a 3-manifold. 
Note that $N$ admits a finite cover $M$ which is orientable.
 We denote by $W=M\cup_{\partial M=\partial M}M$ the double of $M$
which is now a closed, orientable  3-manifold.
Note that the `folding map' $W\to M$ is a retraction onto $M$. This implies in particular that the folding induces a homomorphism $\pi_1W\to \pi_1M$ which is a left inverse to the inclusion induced map $\pi_1M\to \pi_1W$.
Furthermore, it is a consequence of the Kneser--Milnor Prime Decomposition theorem
that $W$ is the connected sum of finitely many prime $3$-manifolds, in particular  $\pi_1W$ is the free product of fundamental groups of closed, orientable, prime 3-manifolds. 

We now write
\[ 
\ba{rcl}
 \MM&:=&\mbox{class of all $3$-manifolds,}\\
 \MM^{\op{or}}&:=&\mbox{class of all orientable $3$-manifolds,}\\
 \ol{\MM}^{\op{or}}&:=&\mbox{class of all closed, orientable $3$-manifolds,}\\
  \ol{\MM}^{\op{or}}_{\op{pr}} &:=&\mbox{class of all closed, orientable,  prime $3$-manifolds}.\ea \]
 We furthermore denote by $\GG,\GG^{\op{or}}, \ol{\GG}^{\op{or}}$ and $\ol{\GG}^{\op{or}}_{\op{pr}} $ the corresponding classes of fundamental groups.
By assumption the Membership Problem is solvable in $\ol{\GG}^{\op{or}}_{\op{pr}} $.

It follows from Theorem \ref{theorem: 3-manifolds are re} that $\ol{\MM}^{\op{or}}_{\op{pr}} $  is recursively enumerable.
Using the 2-skeleton of a triangulation one can write down a presentation for the fundamental group of any
finite connected simplicial complex. It follows that $\ol{\GG}^{\op{or}}_{\op{pr}} $ is also recursively enumerable.
It is now a consequence of  the aforementioned Kneser--Milnor Prime Decomposition theorem and 
from Corollary \ref{cor:solvablempproblemandfreeproducts} that the Membership Problem is solvable for 
$\ol{\GG}^{\op{or}}$.

It follows from the above discussion that any group in $\GG$ is virtually isomorphic to a retract of a group
in $\ol{\GG}^{\op{or}}$. Since any group in $\ol{\GG}^{\op{or}}$ is the free product of finitely many groups in 
$\ol{\GG}^{\op{or}}_{\op{pr}} $, and since $\ol{\GG}^{\op{or}}_{\op{pr}} $ is recursively enumerable it now follows that 
$\ol{\GG}^{\op{or}}$ is also recursively enumerable.
It now follows from Lemma \ref{lem:mpvirtualretract} that the Membership Problem
is solvable in $\GG$. 
\end{proof}

\subsection{The JSJ decomposition of 3-manifolds}\label{section:jsj}

In this paper, by a hyperbolic 3-manifold we mean a compact, orientable 3-manifold with empty or toroidal boundary,
such that the interior admits a complete hyperbolic structure. By a Seifert fibred manifold we always mean an orientable Seifert fibred manifold.

If $N$ is a closed, orientable, prime 3-manifold
then the  Geometrization Theorem of 
Thurston \cite{Th82} and Perelman \cite{Pe02,Pe03a,Pe03b}
says that there exists a collection of incompressible tori such that $N$ cut along the tori consists of components which are either Seifert fibred or hyperbolic. A minimal collection of such tori is furthermore unique up to isotopy.
The elements of  a minimal collection of such tori are called the \emph{JSJ tori} of $N$ and the components of $N$ cut along the JSJ tori are the \emph{JSJ components of $N$}.
The Geometrization Theorem thus in particular shows that  $\pi_1(N)$ is the fundamental group of a graph of groups where the vertex groups are fundamental groups of hyperbolic 3-manifolds and Seifert fibred manifolds and where the edge groups correspond to boundari tori of the JSJ components.

We now consider the following classes of decorated groups:
\[ \ba{rcl} \PP_{\op{hyp}}&=&
\ba{c}\mbox{all decorated groups isomorphic to  $(\pi_1(N),\{\pi_1(T_i)\}_{i\in I})$, where}
\\ \mbox{$N$ is a  hyperbolic 3-manifold and $\{T_i\}_{i\in I}$ are the boundary tori of $N$},\ea
\\[2mm]
\PP_{\op{sfs}}&=&
\ba{c}\mbox{all decorated groups isomorphic to   $(\pi_1(N),\{\pi_1(T_i)\}_{i\in I})$, where}\\
\mbox{$N$ is a Seifert fibered 3-manifold  and $\{T_i\}_{i\in I}$ are the boundary tori of $N$},\ea\\[2mm]
\PP_{\op{cl}}&=&\ba{c}\mbox{all decorated groups isomorphic to  $(\pi_1(N),\emptyset)$, where}\\
\mbox{$N$ is a closed 3-manifold which is hyperbolic or Seifert fibered.}\ea\ea \]
(Note that the definitions of these classes do not depend on the choice of base points and path connecting the base points.)
We also write $\PP=\PP_{\op{hyp}}\cup \PP_{\op{sfs}}\cup \PP_{\op{cl}}$. 

We recall that in order  to prove Theorem \ref{thm:mp} it suffices by Proposition \ref{prop:2}  to show that the Membership Problem is solvable for the class of  fundamental groups 
of closed, orientable,  prime $3$-manifolds.
By the Geometrization Theorem it therefore suffices to prove the Membership Problem for the class  of fundamental groups of graphs of groups which are based on $\PP$.

We now argue that $\PP$ satisfies Properties I to VI from Theorem \ref{thm:kmw05}.
\bn
\item[(I)] In Theorems 
\ref{thm:sfsdoublecosetseparable} and  \ref{thm:hypdoublecosetseparable} we will see that fundamental groups of Seifert fibred manifolds and hyperbolic 3-manifolds are double coset separable, i.e.\ any product $GH$ of finitely generated groups $G$ and $H$ is separable. An argument as in the proof of Lemma \ref{lem:resfinite}
now shows that $\PP$ satisfies Condition I.
\item[(II)] The proof of Condition I also shows that Condition II holds. 
\item[(III)] It is obvious that Condition III holds.
\item[(IV)] The Membership problem in $\Z^2$ 
can be solved (for any finite presentation) using basic algebra. It thus follows that $\PP$ satisfies Condition IV.
\item[(V)] In Proposition \ref{prop:intersectionsboundary} we  will show that $\PP$ satisfies Condition V.
\en

It suffices, therefore, to verify Condition VI, which is a consequence of the following theorem of Jaco--Tollefson \cite{jaco_algorithms_1995}, also proved in Jaco--Letscher--Rubinstein \cite{JLR02}.

\begin{theorem}\label{thm:hypsfsenumerable}
The classes $\PP_{\op{hyp}}, \PP_{\op{sfs}}$ and $\PP_{\op{cl}}$ are recursively enumerable.
In particular $\PP=\PP_{\op{hyp}}\cup \PP_{\op{sfs}}\cup \PP_{\op{cl}}$ is recursively enumerable.
\end{theorem}

\begin{proof}
First note that any hyperbolic 3-manifold and any Seifert fibred manifold appears as the JSJ component 
of a closed, orientable, prime 3-manifold.
By Theorem \ref{theorem: 3-manifolds are re} the class of closed, orientable, prime 3-manifolds is recursively enumerable. For each such triangulated 3-manifold we can determine the JSJ components 
using the algorithm of  Jaco--Tollefson \cite{jaco_algorithms_1995} and also Jaco--Letscher--Rubinstein \cite{JLR02}.
We can thus recursively enumerate the class of all JSJ components of closed, orientable, prime 3-manifolds,
in particular, by the above, we can recursively enumerate the class  of 3-manifolds which are either hyperbolic or  Seifert fibred.

For the manifolds with non-trivial boundary we can furthermore by \cite[Algorithm~8.1]{jaco_algorithms_1995} 
 decide whether or not the  3-manifold is Seifert fibred.

The theorem is now an straightforward consequence of the above algorithms.
\end{proof}

\subsection{Subgroups of fundamental groups of Seifert fibred manifolds and of hyperbolic 3-manifolds}\label{section:subgroupshypsfs}
In this section we will recall well-known results about subgroups of  fundamental groups of Seifert fibred manifolds and of hyperbolic 3-manifolds. We will in particular see that the class of decorated groups $\PP$ from Section \ref{section:jsj} satisfies Conditions I and II. We will also need some of the results from this section in the next section when we deal with Condition V. 

In the following, given a subgroup $\Gamma $ of a group $\pi$ we say that \emph{$\Gamma$ is a virtual retract of $\pi$} if there exists a finite-index subgroup $\pi_0$ which contains $\Gamma$ and such that $\Gamma$ is a retract of $\pi_0$. 

The following theorem is proved implicitly by Scott \cite{Sc78}.

\begin{theorem}\label{thm:surfaceretract}
Any finitely generated subgroup of a surface group is a virtual retract.
\end{theorem}

Note that finitely generated subgroups which are virtual retracts are
in particular separable, see e.g. \cite[(G.10)]{AFW12}. The above theorem thus implies that surface groups are subgroup separable. For the record we note that this furthermore implies the following theorem;
see \cite{Sc78,Sc85} for details.

\begin{theorem}\label{thm:sfslerf}
Fundamental groups of Seifert fibred manifolds are subgroup separable.
\end{theorem}

In fact a somewhat stronger statement holds true. In order to state the result we recall that a group $\pi$ is  called \emph{double-coset separable} if given any finitely generated subgroups $G,H\subset \pi$ the product $GH\subset \pi$ is separable. 
The following theorem was  proved by Niblo \cite{Ni92} building on the aforementioned work of  Scott \cite{Sc78}. 

\begin{theorem}\label{thm:sfsdoublecosetseparable}
The fundamental group of any Seifert fibred manifold is double-coset separable.
\end{theorem}

We now turn to the study of fundamental groups of hyperbolic 3-manifolds.  The first key result is the  Tameness Theorem of Agol \cite{Ag04} and Calegari--Gabai \cite{CG06}:

\begin{theorem}\label{thm:tameness} 
Let $N$ be a hyperbolic $3$-manifold and $\Gamma\subset \pi:=\pi_1(N)$ a finitely generated subgroup. Then precisely one of the following holds:
\bn
\item either $\Gamma$ is a relatively quasiconvex subgroup  of $\pi$, or
\item there exists a finite-index subgroup $\pi_0$ of $\pi$ which contains $\Gamma$ as a normal subgroup with $\pi_0/\Gamma\cong \Z$.
\en
\end{theorem}

We will not be concerned with the precise definition of `relatively quasiconvex'. We will use Theorem \ref{thm:tameness} only in conjunction with the following theorem,
which is a consequence of work of Haglund \cite{Hag08} and 
the Virtually Compact Special Theorem which was proved by  Wise \cite{Wi12a,Wi12b}
for hyperbolic 3-manifolds with boundary and by Agol \cite{Ag13} for closed hyperbolic 3-manifolds.
(See also \cite{CDW12} or \cite{SaW12}. We refer to \cite{AFW12} for  details and precise references.)

\begin{theorem}\label{thm:hypsubgroup}
Let $N$ be a hyperbolic $3$-manifold and $\Gamma\subset \pi:=\pi_1(N)$ a finitely generated subgroup. 
If $\Gamma$ is a relatively quasiconvex subgroup of $\pi$, then $\Gamma$ is a virtual retract of $\pi$.
\end{theorem}

The combination of Theorems \ref{thm:tameness} and \ref{thm:hypsubgroup} implies that the fundamental group of a hyperbolic 3-manifold is subgroup separable; we refer to \cite[(G.10)~and~(G.11)]{AFW12} for details.  In fact, by work of Wise and Hruska, the following stronger result holds.

\begin{theorem}\label{thm:hypdoublecosetseparable}
The fundamental group of any hyperbolic 3-manifold is double-coset separable.
\end{theorem}
\begin{proof}
Let $N$ be a hyperbolic 3-manifold and let $G,H$ be finitely generated subgroups of $\pi=\pi_1N$.  We need to prove that the double-coset $GH$ is separable.  Suppose, therefore, that $g\in\pi\smallsetminus GH$.

Suppose first that $G$ is a virtual fibre in $\pi$.  Replacing $\pi$ by a subgroup of finite index, we may assume that we have an epimorphism
\[
\eta:\pi\to\Z
\]
with kernel $G$.   Then $\eta(g)\notin\eta(H)$.  For any $n$ such that $\eta(g)\notin n\Z$ by $\eta(H)\subseteq n\Z$, the concatenation
\[
\pi\stackrel{\eta}{\to}\Z\to\Z/n\Z
\]
separates $g$ from $GH$, as required.

We may therefore assume that both $G$ and $H$ are relatively quasiconvex in $\pi$.  By a result of Hruska  \cite[Corollary~1.6]{Hr10}, $G$ and $H$ are both quasi-isometrically embedded, and therefore the double coset $GH$ is separable by a theorem of Wise \cite[Theorem~16.23]{Wi12a}.
\end{proof}

\subsection{Computing generating sets for intersections}  \label{section:proofcomputegeneratingssets}

In order to prove Theorem \ref{thm:mp} it now suffices to show that the class of decorated groups $\PP$ which we introduced in Section \ref{section:jsj} satisfies Condition V. We will deal with this issue in this section.

\begin{lemma}\label{lem: Intersections for hyperbolic manifolds}
There exists an algorithm which takes as input a 
 finite decorated group presentation $\Pi=(\ll A\mid R\rr,\{\ll X_i\mid S_i\rr, f_i\}_{i\in I})$ and a finite subset $Y$ of $F(A)$ and which,
 if $\Pi$ represents a decorated group in $\PP_{\op{hyp}}$,  gives for each $i\in I$ as output a finite generating set for $\ll f_i(X_i) \rr\cap \ll Y\rr$  as a subgroup of $\ll A\mid R\rr$.
\end{lemma}

\begin{proof}
Let  $\Pi=(\ll A\mid R\rr,\{\ll X_i\mid S_i\rr, f_i\}_{i\in I})$ be a  finite decorated group presentation  and let $Y$ be a finite subset of $F(A)$. 
We suppose that $(\ll A\mid R\rr,\{\ll f_i(X_i)\rr\}_{i\in I})$ is isomorphic to $(\pi_1(N),\{\pi_1(T_i)\}_{i\in I})$ where $N$ is a hyperbolic 3-manifold and $T_i,i\in I$ are the boundary components of $N$. Let $i\in I$. We write $P=\ll f_i(X_i)\rr$ and $\Gamma=\ll Y \rr\subset \pi$. 
By Theorems \ref{thm:tameness} and \ref{thm:hypsubgroup} precisely  one of the following happens:
\bn
\item[(a)] either there exists either a
 finite-index subgroup $\pi_0$ of $\pi$ and a retraction $r\colon \pi_0\to\Gamma$, or
\item[(b)] there exists a finite-index subgroup $\pi_0$ and a homomorphism $p\colon \pi_0\to \Z$ such that $\Gamma=\ker p$.
\en
 In the former case the algorithms of Lemmas \ref{lem:allfiniteindexsubgroups} and \ref{lem:findretract}
 will find such $\pi_0$ and $r$. In the latter case again 
 a na\"ive search using the Reidemeister--Schreier algorithm together with
 Lemma \ref{lem:givesdesiredquotient} will find such $\pi_0$ and $p$. 
In either case, by Lemma \ref{lem:finiteindexkernelcokernel}  we can  compute generators for $P_0=\pi_0\cap P$.

In  case (b) we have $\Gamma\cap P=\ker p|_ {P_0}$ which can be computed by standard linear algebra.
In case (a) we note that $\Gamma\cap P=r(P_0)\cap P_0$.   Using the solution to the word problem in $\pi$ (see Lemma \ref{lem:resfinite}),  we can determine whether or not all generators of $r(P_0)$ and $P_0$ commute, i.e.\ we can determine whether or not $[r(P_0),P_0]=1$. 

First suppose that  $[r(P_0),P_0]=1$. Recall that $P_0$ is the fundamental group of a boundary torus of a hyperbolic 3-manifold. It is  well-known, see e.g. \cite[Theorem~3.1]{AFW12}, that this implies that $P_0$ is a maximal abelian  subgroup of $\pi_0$. It now follows that $r(P_0)\subseteq P_0$, which implies that  $\Gamma\cap P=r(P_0)$. 

Now suppose that  $[r(P_0),P_0]\ne 1$. The fact that $N$ is hyperbolic implies
by \cite[Corollary 3.11]{AFW12} that   the centralizer of any non-identity element in $\pi_0$ is abelian. It now follows that $r(P_0)\cap P_0=1$ and so $\Gamma\cap P=1$.
\end{proof}

We now also consider the following class of decorated groups:
\[ \PP_{\op{product}}=
\ba{c}\mbox{all decorated groups isomorphic to  $(\pi_1(S^1\times \Sigma),\{\pi_1(T_i)\}_{i\in I})$, where}
\\ \mbox{$\Sigma$ is a  surface and $T_i,i\in I$ are the boundary components of $S^1\times \Sigma$.}\ea\]
We then have the following lemma.

\begin{lemma}\label{lem: Intersections for product manifolds}
There exists an algorithm which takes as input a 
 finite decorated group presentation $\Pi=(\ll A\mid R\rr,\{\ll X_i\mid S_i\rr, f_i\}_{i\in I})$ and a finite subset $Y$ of $F(A)$ and which,
 if $\Pi$ represents a subdecoration for a decorated group in $\PP_{\op{product}}$,  gives for each $i\in I$ as output a finite generating set for $\ll f_i(X_i) \rr\cap \ll Y\rr$  as a subgroup of $\ll A\mid R\rr$.
\end{lemma}

\begin{proof}
Let  $(\ll A\mid R\rr,\{\ll X_i\mid S_i\rr, f_i\}_{i\in I})$ be a finite decorated group presentation and $Y$ a  finite subset  of $F(A)$. We suppose that $(\ll A\mid R\rr,\{\langle f_i(X_i)\rr\}_{i\in I})$ is isomorphic to $(\pi_1(S^1\times \Sigma),\{\pi_1(T_i)\}_{i\in I})$ where  $\Sigma$ is a surface and $T_i, i\in I$ are some  boundary components of $S^1\times \Sigma$. Let $i\in I$.
We write  $\pi=\ll A\mid R\rr$, $P=\ll f_i(X_i)\rr$ and $\Gamma=\ll Y\rr$.

By Theorem \ref{thm:surfaceretract}  every finitely generated subgroup of the  surface group $\pi_1(\Sigma)$ is a virtual retract.  It follows easily that every finitely generated subgroup of $\pi_1(S^1\times \Sigma)=\Z\times \pi_1(\Sigma)$ is a virtual retract.  Therefore, the algorithms of 
Lemmas \ref{lem:allfiniteindexsubgroups} and \ref{lem:findretract} will find a finite-index subgroup $\pi_0$ of $\pi$ and a retraction $r:\pi_0\to\Gamma$.  
As in the proof of Lemma \ref{lem: Intersections for hyperbolic manifolds} we can  compute generators for $P_0=\pi_0\cap P$.

Again, we note that $\Gamma\cap P=r(P_0)\cap P_0$.  An explicit computation again determines whether or not $[r(P_0),P_0]=1$. 
  If so then, just as before, because $P_0$ is maximal abelian we have $r(P_0)\subseteq P_0$ and so $r(P_0)=\Gamma\cap P$.  If not then, by the commutative transitivity of $\pi_1(\Sigma)$, we deduce that $\Gamma\cap P=r(P_0)\cap P_0$ is contained in the centre $Z_0$ of $\pi_0$ and so it suffices to compute $r(P_0)\cap Z_0$.  But now $r(P_0)\cap Z_0$ can be seen in the abelianization of $\pi_0$, and so can be computed by elementary linear algebra.
\end{proof}

\begin{lemma}\label{lem: Intersections for SF manifolds}
There exists an algorithm which takes as input a 
 finite decorated group presentation $\Pi=(\ll A\mid R\rr,\{\ll X_i\mid S_i\rr, f_i\}_{i\in I})$ and a finite subset $Y$ of $F(A)$ and which,
 if $\Pi$ represents a decorated group in $\PP_{\op{sfs}}$,  gives for each $i\in I$ as output a finite generating set for $\ll f_i(X_i) \rr\cap \ll Y\rr$  as a subgroup of $\ll A\mid R\rr$.
\end{lemma}

\begin{proof}
Let  $(\ll A\mid R\rr,\{\ll X_i\mid S_i\rr, f_i\}_{i\in I})$ be a finite decorated group presentation and $Y$ a  finite subset  of $F(A)$. We suppose  that $(\ll A\mid R\rr,\{\langle f_i(X_i)\rr\}_{i\in I})$ is isomorphic to $(\pi_1(N),\{\pi_1(T_i)\}_{i\in I})$ where $N$ is a  Seifert fibered space and $\{T_i\}_{i\in I}$ are the boundary components of $N$. Let $i\in I$. We write  $\pi=\ll A\mid R\rr$, $P=\ll f_i(X_i)\rr$ and $\Gamma=\ll Y\rr$.

By \cite[Theorem~11.10]{He76} there exists a finite cover of $N$ which is a product $S^1\times \Sigma$.
Using Lemma \ref{lem:allfiniteindexsubgroups} we now enumerate all finite-index subgroups of $\pi$.
We can furthermore enumerate  all fundamental groups of products $S^1\times \Sigma$, $\Sigma$ a surface,
and using an obvious generalization of Lemma \ref{lem:findiso}
we will eventually find a finite-index subgroup $\pi_0$ of $\pi$, a presentation $\ll A\mid R\rr$
of $\pi_1(S^1\times \Sigma)$, $\Sigma$ a surface,
and an isomorphism $g\colon \pi_0\to \Gamma:=\ll A\mid R\rr$ such that $g(P\cap \pi_0)$ is the fundamental group of a boundary component of $S^1\times \Sigma$.

Using Lemma \ref{lem:finiteindexkernelcokernel} we can find 
 a generating set $Y_0$ for $\Gamma\cap \pi_0$, a finite presentation $\ll X_0\mid S_0\rr$,
 and an isomorphism $f_0\colon \ll X_0\mid S_0\rr\to P \cap \pi_0$.
We now apply the algorithm of Lemma \ref{lem: Intersections for product manifolds} to the finite decorated group presentation 
$(\ll A\mid R\rr, \{\ll X_0\mid S_0\rr,f_0\})$ and the finite set $Y_0$.
The algorithm then gives us a  generating set
for $(\Gamma \cap \pi_0)\cap (P\cap \pi_0)=(\Gamma \cap \pi_0)\cap P$. 

Note that $(\Gamma \cap \pi_0)\cap P$ is a finite index subgroup of $\Gamma\cap P$, which in turn is a subgroup of $P\cong \Z^2$. It is now straightforward to list the (finitely many) subgroups of $P\cong \Z^2$ 
which contain $(\Gamma\cap \pi_0)\cap P$ as a finite index subgroup. For each of these subgroups 
we pick a finite number of generators and using the fact that $\pi$ is subgroup separable,
see Theorem \ref{thm:sfslerf},  we can check whether the generators lie in $P$ and in $\Gamma$. 
\end{proof}

We  are now ready to prove that  the class $\PP$ of decorated  groups satisfies Condition V. More precisely,
we have the following proposition.

\begin{proposition}\label{prop:intersectionsboundary}
There exists an algorithm which takes as input a 
 finite decorated group presentation $\Pi=(\ll A\mid R\rr,\{\ll X_i\mid S_i\rr, f_i\}_{i\in I})$ and a finite subset $Y$ of $F(A)$ and which,
 if $\Pi$ represents a decorated group in $\PP$,  gives for each $i\in I$ as output a finite generating set for $\ll f_i(X_i) \rr\cap \ll Y\rr$  as a subgroup of $\ll A\mid R\rr$.
\end{proposition}

\begin{proof}
 Let $(\ll A\mid R\rr,\{\ll X_i\mid S_i\rr, f_i\}_{i\in I})$  be a 
    finite decorated group presentation which represents a decorated group in $\PP$ and let  $Y\subset F(A)$ be a finite set. Let $i\in I$.
By the solution to the word problem for $\ll A\mid R\rr$ we can determine whether or not $f_i(X_i)$ generates the trivial group. If it does, then there is nothing to show. 

Now suppose that $f_i(X_i)$ does not generate the trivial group. By definition of $\PP$ it follows that $\Pi$
represents an element in $\PP_{\op{hyp}}$ or it represents an element in $\PP_{\op{sfs}}$.

Using Theorem \ref{thm:hypsfsenumerable} and an obvious generalization of Lemma \ref{lem:findiso} 
we can now  certify that  $(\ll A\mid R\rr, \{\ll f_i(X_i)\rr\}_{i\in I})$ is isomorphic to a subdecoration  of a decorated group in $\PP_{\op{hyp}}$, or we can certify that $(\ll A\mid R\rr,X)$ is isomorphic to a subdecoration of a decorated group in $\PP_{\op{sfs}}$.
(It follows from basic facts in 3-manifold topology that only one of the two cases can occur, but this fact is irrelevant for the proof of this proposition.)
 In the former case we now apply the algorithm
from Lemma  \ref{lem: Intersections for hyperbolic manifolds}, while in the latter case we apply the algorithm from Lemma \ref{lem: Intersections for SF manifolds}
\end{proof}

\section{Alternative approaches and open questions}\label{section:alternative}

In the proof of our main theorem we used two big theorems on fundamental groups of hyperbolic 3-manifolds: the Tameness Theorem of Agol \cite{Ag04} and Calegari--Gabai \cite{CG06}
and Theorem \ref{thm:hypsubgroup} which is a consequence of the Virtually Compact Special Theorem of Agol \cite{Ag13} and Wise \cite{Wi12a,Wi12b}.  The Tameness Theorem is indispensable: it is needed to control geometrically infinite subgroups.  However, it is quite possible that one could also  prove Theorem \ref{mainthm} without appealing to  the  Virtually Compact Special Theorem.  
For example, Gitik \cite{Gi96} and Kapovich \cite{Ka96}
showed that the Membership Problem is solvable for quasiconvex subgroups of word-hyperbolic groups.
It is now straightforward to see that one can prove our main theorem by appealing to this result
and to appealing to  Theorem \ref{thm:hypsubgroup} only for hyperbolic 3-manifolds with non-empty boundary.
It is now an interesting question whether one can also replace Theorem \ref{thm:hypsubgroup} by more general methods from geometric group theory.

In the following we consider the class of decorated presentations of \emph{toral relatively hyperbolic groups}. That is, we consider
\[ \HH=
\ba{c}\mbox{all decorated groups isomorphic to  $(\pi,\{\Gamma_i\}_{i\in I})$, where $\pi$ is a group which}
\\ \mbox{is\,hyperbolic\,relative\,to\,the\,finite\,collection\,of\,\,f.g.\,abelian\,subgroups\,$\{\Gamma_i\}_{i\in I}$.}\ea \]
It is now fairly straightforward to see that one can prove our main theorem without referring to Theorem \ref{thm:hypsubgroup} if one can give affirmative answers to the following three questions.

The first is a generalization of the above-mentioned work of Gitik and Kapovich to the toral relatively hyperbolic setting.

\begin{question}\label{qu: RQC membership}
Does there exist an algorithm which takes as input a 
 finite decorated group presentation $\Pi=(\ll A\mid R\rr,\{\ll X_i\mid S_i\rr, f_i\}_{i\in I})$ and a finite subset $Y$ of $F(A)$ and which,
 if $\Pi$ represents a  decorated group in $\HH$
 and if $Y$  generates a relatively quasi-convex subgroup of $\ll A\mid R\rr$, determines whether or not a given element in $\ll A\mid R\rr$
 lies in  $\langle Y\rangle\subset \ll A\mid R\rr$?
\end{question}

The second question generalizes the first question to double cosets.

\begin{question}\label{qu: RQC double coset membership}
Does there exist an algorithm which takes as input a finite decorated group presentation $\Pi=(\ll A\mid R\rr,\{\ll X_i\mid S_i\rr, f_i\}_{i\in I})$, a finite subset $Y$ of $F(A)$ and an index $i\in I$ and which, if $\Pi$ represents a  decorated group in $\HH$ and if $Y$  generates a relatively quasi-convex subgroup of $\ll A\mid R\rr$, determines whether or not a given element in $\ll A\mid R\rr$  lies in the double coset $\langle f_i(X_i)\rangle \langle Y\rangle\subset \ll A\mid R\rr$?
\end{question}

The final question asks for an algorithm to compute the intersection of a relatively quasiconvex subgroup and a maximal parabolic subgroup.

\begin{question}\label{qu: RQC intersection}
Does there exist an algorithm which takes as input a finite decorated group presentation $\Pi=(\ll A\mid R\rr,\{\ll X_i\mid S_i\rr, f_i\}_{i\in I})$ and a finite subset $Y$ of $F(A)$ and which, if $\Pi$ represents a  decorated group in $\HH$,  gives for each $i\in I$ as output a finite generating set for $\ll f_i(X_i) \rr\cap \ll Y\rr$  as a subgroup of $\ll A\mid R\rr$?
\end{question}

\end{document}